\documentclass[10pt]{article}
\font\smallit=cmti10
\font\smalltt=cmtt10

\usepackage{caption}
\usepackage{color}
\usepackage{amssymb,latexsym,amsmath,epsfig,amsthm} 
\usepackage{empheq}
\usepackage{here}
\usepackage{url}
\usepackage{array}
\usepackage{hhline}
\newcolumntype{!}{!{\vrule width 1.5pt}}
\usepackage{ascmac}
\makeatletter
\usepackage{here}
\usepackage{comment}

\renewcommand\section{\@startsection {section}{1}{\z@}
{-30pt \@plus -1ex \@minus -.2ex}
{2.3ex \@plus.2ex}
{\normalfont\normalsize\bfseries\boldmath}}

\renewcommand\subsection{\@startsection{subsection}{2}{\z@}
{-3.25ex\@plus -1ex \@minus -.2ex}
{1.5ex \@plus .2ex}
{\normalfont\normalsize\bfseries\boldmath}}

\renewcommand{\@seccntformat}[1]{\csname the#1\endcsname. }

\makeatother
\newtheorem{theorem}{Theorem}
\newtheorem{lemma}{Lemma}
\newtheorem{conjecture}{Conjecture}

\theoremstyle{definition}
\newtheorem{definition}{Definition}

\newtheorem{example}{Example}

\begin{document}

\begin{center}
\uppercase{\bf A Subtraction Nim with a Pass}
\vskip 20pt
{\bf Urban Larsson}\\
{\smallit IEOR, IIT Bombay, India}\\
{\tt larsson@iitb.ac.in}
\vskip 10pt
{\bf Hikaru Manabe}\\
{\smallit University of Tuskuba, Tsukuba City, Japan}\\
{\tt urakihebanam@gmail.com}
\vskip 10pt
{\bf Ryohei Miyadera }\\
{\smallit Keimei Gakuin Junior and High School, Kobe City, Japan}\\
{\tt runnerskg@gmail.com}

\end{center}


\centerline{\bf Abstract}
\noindent
We consider a subtraction Nim with subtraction set $\{s_1,s_2,s_3\}=\{2,4n,4n+2\}$, where 
$n$ is a positive integer such that $n \geq 3$. We do not treat the case that $n=1$ or $n=2$ in this article.
We show that this game satisfies the reverse-mex property of Grundy numbers, i.e., $\mathcal{G}(x)=\operatorname{mex}\{\mathcal{G}(x+s_1), \mathcal{G}(x+s_2), \mathcal{G}(x+s_3)\}$, where the mex is taken over successors rather than predecessors.
We modify the rule of this subtraction Nim to allow a one-time pass, that is, a passing move usable at most once during the game, unavailable from terminal positions; once used by either player, it becomes unavailable.
In classical Nim, the introduction of a pass move complicates the game, and finding a formula that describes the set of P-positions in traditional three-pile Nim with a pass remains an important open question.
In the case of subtraction Nim with a pass, however, the introduction of a pass move does not complicate the game.
We prove that this game still satisfies the reverse-mex property of Grundy numbers when a pass move is available.
 \pagestyle{myheadings} 
 \markright{\smalltt \hfill} 
 \thispagestyle{empty} 
 \baselineskip=12.875pt 
 \vskip 30pt

\section{Introduction}
Subtraction games have a long history of combinatorial game research, 
with early results by Golomb \cite{golomb} and Winning Ways \cite{winningway1}, 
but new facts have still been discovered. 
See \cite{larsson3}, \cite{larsson1}, and \cite{larssonBhagat}.

We focus on the following particular game for the reasons described below.

\begin{definition}\label{defofgame}
There is a pile of stones, and the number of stones to be removed belongs 
to the subtraction set $\{s_1, s_2, s_3\}=\{2,4n,4n+2\}$, 
where $n$ is a positive integer with $n \geq 3$. 
The player who cannot move is the loser. 
\end{definition}

We focus on this particular game in Definition \ref{defofgame} for the following two reasons.
The first reason is that this game exemplifies the reverse-mex property of Grundy numbers, 
which is a remarkable property shared by some subtraction Nim games.
The second reason is that the proofs will be extremely complicated if we attempt to prove 
the results in this paper for subtraction games in greater generality.
Subtraction Nims are very sensitive to changes in conditions; 
even for the simple case $\{s_1,s_2,s_3\}=\{2,4n,4n+2\}$, 
the two different assumptions $n \geq 3$ and $n=1,2$ produce different formulas.

The proof method in this paper is elementary in nature, relying on case analysis and induction. 
Nevertheless, shortening the proof by more sophisticated methods appears unlikely, 
since the sequence of Grundy numbers $\mathcal{G}(x)$ for $x=0,1,\dots$ 
enters a loop only after taking irregular values for $x=0,1,\dots, 12n+8$.

The rest of this paper is organized as follows. 
In Section \ref{subtractionnim}, we analyze the Grundy numbers of the subtraction Nim 
defined in Definition~\ref{defofgame} and prove the reverse-mex property. 
In Section \ref{subtractionnimpass}, we introduce a one-time pass move into the game and prove that 
the reverse-mex property is preserved. 
In Section \ref{conject}, we present conjectures on reverse-mex properties 
of subtraction Nim and subtraction Nim with a pass.

Let $\mathbb{Z}_{\ge 0}$ and $\mathbb{N}$ represent the sets of non-negative integers and natural numbers, respectively.

\begin{definition}\label{NPpositions}
	$(a)$ A position is referred to as a $\mathcal{P}$-\textit{position} if it is the winning position for the previous player (the player who has just moved), as long as they play correctly at each stage. \\
	$(b)$ A position is referred to as an $\mathcal{N}$-\textit{position} if it is the winning position for the next player, as long as they play correctly at each stage.
\end{definition}

\begin{definition}\label{defofmove}
	For any position $\mathbf{p}$ in game $\mathbf{G}$, a set of positions can be reached by a single move in $\mathbf{G}$, which we denote as \textit{move}$(\mathbf{p})$. 
\end{definition}

\begin{definition}\label{defofmexgrundy}
$\mathrm{(i)}$ The \textit{minimum excluded value} ($\textit{mex}$) of a set $S$ of nonnegative integers is the least nonnegative integer that is not in S. \\
$\mathrm{(ii)}$ Let $\mathbf{p}$ be a position in the impartial game. The associated \textit{Grundy number} is denoted by $\mathcal{G}(\mathbf{p})$ and is 
 recursively defined by 
	$\mathcal{G}(\mathbf{p}) = \textit{mex}(\{\mathcal{G}(\mathbf{h}):  \mathbf{h} \in move(\mathbf{p})\}).$
\end{definition}

\begin{theorem}[\cite{lesson}]\label{theoremofsumg}
For any position $\mathbf{g}$ in $\mathbf{G}$, 
	$G_{\mathbf{G}}(\mathbf{g})=0$ if and only if $\mathbf{g}$ is the $\mathcal{P}$-position. 
\end{theorem}

For the details of combinatorial game theory, see \cite{lesson} and \cite{combysiegel}.
In the following sections, we consider subtraction Nim and subtraction Nim with a pass.
For the research on subtraction Nim, see \cite{winningway1}, and for the research on traditional Nim with a pass, see \cite{nimpass2} and \cite{nimpassfinite}.
For the research on combinatorial games with a pass, see \cite{regularity},
\cite{integers1}, \cite{integer2025}, and \cite{integer2023}.
\section{Subtraction Game}\label{subtractionnim}

We define the subtraction Nim that we consider in this paper.
\begin{definition}\label{defofmove}
There is a pile of stones, and the number of stones to be removed belongs to the subtraction set $\{s_1, s_2, s_3\}=\{2,4n,4n+2\}$ such that $n \geq 3$.
The player who cannot move is the loser. Then,
$move(x)=\{x-s_i:x-s_i \geq 0 \text{ and } i = 1,2,3\}$.
\end{definition}

\begin{definition}\label{defofgrundy}
(i) When the number of stones in the pile is $x$, 
we denote by $(x)$ the position of the game.\\
(ii) We denote by $\mathcal{G}(x)$ the Grundy number of the position $(x)$.
\end{definition}

\begin{definition}\label{defofgrundy}
(i) When the number of stones in the pile is $x$, 
we denote by $(x)$ the position of the game.\\
(ii) We denote by $\mathcal{G}(x)$ the Grundy number of the position $(x)$.
\end{definition}

By Definitions \ref{defofmexgrundy}, \ref{defofgrundy} and \ref{defofmove}, we have
$\mathcal{G}(x) = mex(\{\mathcal{G}(x-s_1), \mathcal{G}(x-s_2), \mathcal{G}(x-s_3)\})$. 
We define the reverse-mex property in Definition \ref{inversemex}.

\begin{definition}\label{inversemex}
A subtraction Nim is said to satisfy \textit{reverse-mex} property if 
$\mathcal{G}(x) = mex(\{\mathcal{G}(x+s_1), \mathcal{G}(x+s_2), \mathcal{G}(x+s_3)\})$. 
\end{definition}

In Figure \ref{wall4m5}, we present the list $\{\mathcal{G}(x):x \in \mathbb{Z}_{\ge 0}\}$.
\begin{figure}[H]
\begin{minipage}[t]{1\textwidth}
\begin{center}
\includegraphics[height=3.4cm]{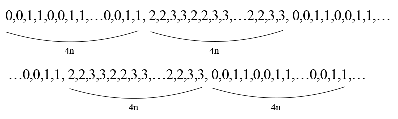}
\caption{ $\{\mathcal{G}(0), \mathcal{G}(1), \mathcal{G}(2), \dots \}$}
\label{wall4m5}
\end{center}
\end{minipage}
\end{figure}

\begin{theorem}\label{lemmafor4mcase}
Let $ n \in \mathbb{N}$ and $p=8n$.
Let $\mathcal{G}(x)$ be the Grundy number of the subtraction Nim with the subtraction set $\{s_1,s_2,s_3\}$ $= \{2, 4n, 4n+2\}$, where $x$ is the number of stones in the pile.
For $k,m \in \mathbb{Z}_{\ge 0}$ such that $m \leq n-1$,
we have the following: \\
(i) for $t=0,1$,
\begin{equation}
\mathcal{G}(pk+4m+t)=0;  \label{case0}
\end{equation}
(ii) for $t=2,3$,
\begin{equation}
\mathcal{G}(pk+4m+t)=1;   \label{case1}
\end{equation}
(iii) for  $t=0,1$,
\begin{equation}
\mathcal{G}(pk+4n+4m+t)=2;  \label{case2}
\end{equation}
(iv) for  $t=2,3$,
\begin{equation}
\mathcal{G}(pk+4n+4m+t)=3.  \label{case3}
\end{equation}
\end{theorem}
\begin{proof} 
We refer to Figure \ref{wall4m5} for the concrete examples of the following proof process.
Since $move(0)=move(1)=\emptyset,$ $\mathcal{G}(0)=\mathcal{G}(1)=0$.
Since $move(2)=\{2-s_1\}=\{0\}$ and $move(3)=\{3-s_1\}=\{1\}$, we obtain $\mathcal{G}(2)=\mathcal{G}(3)=1$.

For $m$ such that $1 \leq m \leq n-1$, 
$move(4m+t)=\{4m+t-s_1\}=\{4m+t-2\}$, and hence 
by using the definition of Grundy number repeatedly, we obtain 
$\mathcal{G}(4m+t)=0$ for $t=0,1$ and $\mathcal{G}(4m+t)=1$ for $t=2,3$.

$move(4n)= \{4n-s_1,4n-s_2\}=\{4n-2,0\}$, $\mathcal{G}(4n-2)=1$ and $\mathcal{G}(0)=0$. Hence $\mathcal{G}(4n)=2$.
Similarly, we obtain $\mathcal{G}(4n+1)=2$.
$move(4n+2)=\{4n+2-s_1,4n+2-s_2, 4n+2-s_3\}=\{4n,2,0\}$, $\mathcal{G}(4n)=2$, $\mathcal{G}(2)=1$, and $\mathcal{G}(0)=0$. Hence  $\mathcal{G}(4n+2)=3$. Similarly,  we obtain $\mathcal{G}(4n+3)=3$.

Suppose that $1 \leq m \leq n-1$. Then,  $4n+4+t \leq 4n+4m+t \leq 8n-4+t$ for $t=0,1,2,3$. We use mathematical induction.

If $t=0,1$, $move(4n+4m+t)=\{4n+4(m-1)+t+2, 4m+t, 4(m-1)+t+2\}$. 
By the assumption of mathematical induction, $\mathcal{G}(4n+4(m-1)+t+2)=3$, $\mathcal{G}(4m+t)=0$, and $\mathcal{G}(4(m-1)+t+2)=1$. Hence, $\mathcal{G}(4n+4m+t)=2$.

If $t=2,3$, $move(4n+4m+t)=\{4n+4m+t-2, 4m+t, 4m+t-2\}$. 
By the assumption of mathematical induction,  $\mathcal{G}(4n+4m+t-2)=2$, $\mathcal{G}(4m+t)=1$, and $\mathcal{G}(4m+t-2)=0$. Hence, $\mathcal{G}(4n+4m+t)=3$.

Let $h \in \mathbb{N}$. We use mathematical induction,and assume that Equations (\ref{case0}), (\ref{case1}), (\ref{case2}), and (\ref{case3}) are valid for 
any $k$ such that $k \leq h$.

For $t=0,1$
$move(ph+t)=\{p(h-1)+4n+4(n-1)+t+2, p(h-1)+4n+t, p(h-1)+4(n-1)+t+2\}$.
Since   $\mathcal{G}(p(h-1)+4n+4(n-1)+t+2)=3$,  $\mathcal{G}(p(h-1)+4n+t)=2$, and $\mathcal{G}(p(h-1)+4(n-1)+t+2)=1$, we obtain $\mathcal{G}(ph+t)=0$.

For $t=2,3$,
$move(ph+t)=\{ph+t-2, p(h-1)+4n+t, p(h-1)+4n+t-2\}$.
Since $\mathcal{G}(ph+t-2)=0$, $\mathcal{G}(p(h-1)+4n+t)=3$, and $\mathcal{G}(p(h-1)+4n+t-2)=2$
, we obtain  $\mathcal{G}(pt+t)=1$.

 Suppose that 
$1 \leq m \leq n-1$. 

For $t=0,1$,
$move(ph+4m+t)=\{ph+4(m-1)+t+2, p(h-1)+4n+4m+t, p(h-1)+4n+4(m-1)+t+2\}$.
Since $\mathcal{G}(ph+4(m-1)+t+2)=1$, $\mathcal{G}(p(h-1)+4n+4m+t)=2$, and $\mathcal{G}(p(h-1)+4n+4(m-1)+t+2)=3$, we obtain $\mathcal{G}(ph+4m+t)=0$.

For $t=2,3$,
$move(ph+4m+t)=\{ph+4m+t-2, p(h-1)+4n+4m+t, p(h-1)+4n+4m+t-2 \}$.
Since $\mathcal{G}(ph+4m+t-2)=0$, $\mathcal{G}(p(h-1)+4n+4m+t)=3$, and $\mathcal{G}(p(h-1)+4n+4m+t-2)=2$,  
we obtain $\mathcal{G}(ph+4m+t)=1$.

For $t=0,1$,
$move(ph+4n+t)=\{ph+4(n-1)+t+2, ph+t, p(h-1)+4n+4(n-1)+t+2\}$.
Since $\mathcal{G}(ph+4(n-1)+t+2)=1$, $\mathcal{G}(ph+t)=0$, and 
$\mathcal{G}(p(h-1)+4n+4(n-1)+t+2)=3$, we obtain $\mathcal{G}(ph+4n+t)=2$.

For $t=2,3$,
$move(ph+4n+t)=\{ph+4n+t-2, ph+t, ph+t-2\}$.
Since $\mathcal{G}(ph+4n+t-2)=2$, $\mathcal{G}(ph+t)=1$, and 
$\mathcal{G}(ph+t-2)=0$, we obtain $\mathcal{G}(ph+4n+t)=3$.

For $t=0,1$,
$move(ph+4n+4m+t)=\{ph+4n+4(m-1)+t+2, ph+4m+t, ph+4(m-1)+t+2\}.$
Since $\mathcal{G}(ph+4n+4(m-1)+t+2)=3$, $\mathcal{G}(ph+4m+t)=0$, and 
$\mathcal{G}(ph+4(m-1)+t+2)=1$, we obtain $\mathcal{G}(ph+4n+4m+t)=2$.

For $t=2,3$,
$move(ph+4n+4m+t)=\{ph+4n+4m+t-2, ph+4m+t, ph+4m+t-2\}.$
Since $\mathcal{G}(ph+4n+4m+t-2)=2$, $\mathcal{G}(ph+4m+t)=1$, and 
$\mathcal{G}(ph+4m+t-2)=0$, we obtain $\mathcal{G}(ph+4n+4m+t)=3$.
\end{proof}

The subtraction Nim with the subtraction set $\{2,4n,4n+2\}$ has the reverse-$mex$ property 
 of Definition \ref{inversemex}, and this fact is 
presented in Theorem \ref{inversemexth}.
\begin{theorem}\label{inversemexth}
For any $x \in \mathbb{N}$,
\begin{equation}
\mathcal{G}(x)=mex(\{\mathcal{G}(x+2), \mathcal{G}(x+4n), \mathcal{G}(x+4n+2)\}).
\end{equation}
\end{theorem}
\begin{proof}
Suppose that $0 \leq m \leq n-1$ and $p \in \mathbb{Z}_{\ge 0}$. We have four cases.\\
\noindent {\tt Case 1}: Let $x=pk+4m+t$ with $t=0,1$. Then, $mex(\{\mathcal{G}(x+2), \mathcal{G}(x+4n), \mathcal{G}(x+4n+2)\})$
$=mex(\{\mathcal{G}(pk+4m+t+2), \mathcal{G}(pk+4n+4m+t),\mathcal{G}(pk+4n+4m+t+2)\})=mex(\{1,2,3\})$
$=0 = \mathcal{G}(x).$\\
\noindent {\tt Case 2}:
Let $x=pk+4m+t$ with $t=2,3$. Then, 
\begin{equation}
\mathcal{G}(x+2)=\mathcal{G}(pk+4m+t+2)=\mathcal{G}(pk+4(m+1)+t-2)=0 \label{grundy11}  
\end{equation}
 if $m \leq n-2$, and 
\begin{equation}
\mathcal{G}(x+2)=\mathcal{G}(pk+4m+t+2)=\mathcal{G}(pk+4n+t-2)=2 \label{grundy22}     
\end{equation}
if $m = n-1$.
\begin{equation}
\mathcal{G}(x+4n)=\mathcal{G}(pk+4n+4m+t)=3. \label{grundy33}     
\end{equation}
\begin{equation}
\mathcal{G}(x+n+2)=\mathcal{G}(pk+4n+4m+t+2)=\mathcal{G}(pk+4n+4(m+1)+t-2)=2 \label{grundy44}   
\end{equation}
if $m \leq n-2$, and 
\begin{equation}
\mathcal{G}(x+n+2)=\mathcal{G}(pk+4n+4m+t+2)=\mathcal{G}(p(k+1)+t-2)=0 \label{grundy55}     
\end{equation}
if $m = n-1$.
By Equations (\ref{grundy11}), (\ref{grundy22}), (\ref{grundy33}), (\ref{grundy44}), and (\ref{grundy55}),
$mex(\{\mathcal{G}(x+2), \mathcal{G}(x+4n), \mathcal{G}(x+4n+2)\})$
$=mex(\{\mathcal{G}(pk+4m+t+2), \mathcal{G}(pk+4n+4m+t),\mathcal{G}(pk+4n+4m+t+2)\})$
$=mex(\{0,3,2\})=1$ or $mex(\{2,3,0\})=1$, and $\mathcal{G}(x)=1$.\\
\noindent {\tt Case 3}: Let $x=pk+4n+4m+t$ with $t=0,1$. Then, 
$mex(\{\mathcal{G}(x+2), \mathcal{G}(x+4n), \mathcal{G}(x+4n+2)\})$
$=mex(\{\mathcal{G}(pk+4n+4m+t+2), \mathcal{G}(p(k+1)+4m+t),\mathcal{G}(p(k+1)+4m+t+2)\})=mex(\{3,0,1\})$
$=2 = \mathcal{G}(x).$ \\
\noindent {\tt Case 4}: Let $x=pk+4n+4m+t$ with $t=2,3$. Then, 
\begin{equation}
\mathcal{G}(x+2)=\mathcal{G}(pk+4n+4m+t+2)=\mathcal{G}(pk+4n+4(m+1)+t-2)=2 \label{grundy66}     
\end{equation}
if $m \leq n-2$, and 
\begin{equation}
\mathcal{G}(x+2)=\mathcal{G}(pk+4n+4m+t+2)=\mathcal{G}(p(k+1)+t-2)=0  \label{grundy77}     
\end{equation}
if $m = n-1$.
\begin{equation}
\mathcal{G}(x+4n)=\mathcal{G}(pk+8n+4m+t)=\mathcal{G}(p(k+1)+4m+t)=1.  \label{grundy88}     
\end{equation}
\begin{equation}
\mathcal{G}(x+4n+2)=\mathcal{G}(pk+8n+4m+t+2)=\mathcal{G}(p(k+1)+4(m+1)+t-2)=0 \label{grundy99} 
\end{equation}
if $m \leq n-2$, and 
\begin{equation}
\mathcal{G}(x+4n+2)=\mathcal{G}(pk+8n+4m+t+2)=\mathcal{G}(p(k+1)+4n+t-2)=2 \label{grundy1010}   
\end{equation}
if $m=n-1$.
By Equations (\ref{grundy66}), (\ref{grundy77}), (\ref{grundy88}), (\ref{grundy99}), and (\ref{grundy1010}),
$mex(\{\mathcal{G}(x+2), \mathcal{G}(x+4n), \mathcal{G}(x+4n+2)\})$
$=mex(\{\mathcal{G}(pk+4n+4m+t+2), \mathcal{G}(pk+8n+4m+t),\mathcal{G}(pk+8n+4m+t+2)\})$
$=mex(\{2,1,0\})=3$ or $mex(\{0,1,2\})=3$, and 
$\mathcal{G}(x)=3.$
\end{proof}

\section{Subtraction Nim with a Pass}\label{subtractionnimpass}
This paper presents the research on combinatorial games with a pass.
In the classical Nim, if the standard rules of the game are modified to allow a one-time pass, that is, a passing move that may be used at most once in the game and not from a terminal position, and once either player has used a pass, it is no longer available, the game becomes complicated. 
To find the formula that describes the set of previous players' winning positions of the traditional three-pile Nim with a pass remains an important open question. 

 \begin{definition}\label{subtractionsmallpass}
In the subtraction Nim in Definition \ref{defofmove}, we allow a pass-move. We denote the position of the game by $(x,p)$, where $x$ is the number of stones in the pile and $p=1$ if a pass is available, and $p=0$ if not.
\end{definition}

In this section, we show that the subtraction Nim in this paper displays remarkable properties when a pass is allowed.
By the definition of Grundy number, $\mathcal{G}(x,1)$ satisfies 
\begin{equation}
\mathcal{G}(x,1)=mex(\{\mathcal{G}(x-2,1), \mathcal{G}(x-4n,1), \mathcal{G}(x-4n-2,1),\mathcal{G}(x,0)\}),\label{calgrundy1}
\end{equation}
but for $x \geq 12n+9$, we obtain 
\begin{equation}
\mathcal{G}(x,1)=mex(\{\mathcal{G}(x-2,1), \mathcal{G}(x-4n,1), \mathcal{G}(x-4n-2,1)\})\label{calgrundy2}
\end{equation}
and 
\begin{equation}
\mathcal{G}(x,1)=mex(\{\mathcal{G}(x+2,1), \mathcal{G}(x+4n,1), \mathcal{G}(x+4n+2,1)\}).\label{calgrundy3}
\end{equation}
By Equations (\ref{calgrundy1}) and (\ref{calgrundy2}), for this subtraction Nim with a pass, the calculation of Grundy number $\mathcal{G}(x,1)$ does not need the value of $\mathcal{G}(x,0)$ when $x \geq 12n+9$.

By Equation (\ref{calgrundy3}), we have the reverse-$mex$ property of Definition \ref{inversemex}.
The subtraction Nim with the subtraction set $\{2,4n,4n+2\}$ has the reverse-$mex$ property 
 of Definition \ref{inversemex}.

for $x \geq 12n+9$.

\begin{table}[H] 
  \centering 
  \caption{Grundy Numbers for $x=0,1,\dots, 4n+3$}
  \label{tables4n3}
{
\setlength{\tabcolsep}{1.5pt} 
\begin{tabular}{|r @{\hspace{0.7pt} \vrule width 1.4pt \hspace{3pt}}r|r|r|r@{\hspace{0.7pt} \vrule width 1.4pt \hspace{3pt}}r|r|r|r@{\hspace{0.7pt} \vrule width 1.4pt \hspace{3pt}}c @{\hspace{0.7pt} \vrule width 1.4pt \hspace{3pt}} r |r|r|r @{\hspace{0.7pt} \vrule width 1.4pt \hspace{0.0pt}} r |r|r|r @{\hspace{0.4pt} \vrule width 1.4pt \hspace{0pt}}} \hline
    $x$ & 0 & 1 & 2 & 3 & 4 & 5 & 6 & 7 & \dots & $4n{-}4$ & $4n{-}3$ & $4n{-}2$ & $4n{-}1$ &$4n$ & $4n{+}1$ & $4n{+}2$ & $4n{+}3$ \\ \hline
    $\mathcal{G}(x,0)$ & 0 & 0 & 1 & 1 & 0 & 0 & 1 & 1 & \dots & 0 & 0 & 1 & 1 & 2 & 2 & 3 & 3  \\ \hline
    $\mathcal{G}(x,1)$ & 0 & 1 & 2 & 0 & 1 & 1 & 0 & 0 & \dots & 1 & 1 & 0 & 0 & 1 & 3 & 4 & 2\\ \hline
\end{tabular}
}
\end{table}

\begin{table}[H] 
  \centering 
  \caption{Grundy Numbers for $x=4n+4,\dots, 8n-1$}
  \label{table4n4} 
{
\setlength{\tabcolsep}{1pt} 
\begin{tabular}{@{\hspace{-3.5pt} \vrule width 1.4pt \hspace{3pt}}r|r|r|r@{\hspace{0.4pt} \vrule width 1.4pt \hspace{3pt}}r|r|r|r@{\hspace{0.4pt} \vrule width 1.4pt \hspace{0.0pt}}r @{\hspace{0.4pt} \vrule width 1.4pt \hspace{3pt}}r|r|r|r@{\hspace{0.4pt} \vrule width 1.4pt \hspace{0pt}}} \hline
 $4n{+}4$ & $4n{+}5$ & $4n{+}6$ & $4n{+}7$ & $4n{+}8$ & $4n{+}9$ & $4n{+}10$ & $4n{+}11$ &  \dots & $8n{-}4$ & $8n{-}3$ & $8n{-}2$ & $8n{-}1$\\ \hline
 2 & 2 & 3 & 3 & 2 & 2 & 3 & 3 & \dots & 2 & 2 & 3 & 3\\ \hline
 0 & 3 & 2 & 2 & 3 & 3 & 2 & 2 & \dots & 3 & 3 & 2 & 2\\ \hline
\end{tabular}
}
\end{table}

\begin{table}[H] 
  \centering 
  \caption{Grundy Numbers for $x=8n,\dots, 8n+11$}
  \label{table8n} 
{
\setlength{\tabcolsep}{1pt} 
\begin{tabular}{@{\hspace{-3.5pt} \vrule width 1.4pt \hspace{3pt}}r|r|r|r@{\hspace{0.7pt} \vrule width 1.4pt \hspace{3pt}}r|r|r|r@{\hspace{0.7pt} \vrule width 1.4pt \hspace{0.0pt}}r|r|r|r@{\hspace{0.7pt} \vrule width 1.4pt \hspace{0.0pt}}  r |}\hline
 $8n$ & $8n{+}1$ & $8n{+}2$ & $8n{+}3$ & $8n{+}4$ & $8n{+}5$ & $8n{+}6$ & $8n{+}7$ & $8n{+}8$ & $8n{+}9$ & $8n{+}10$ & $8n{+}11$ &  \dots \\ \hline
 0 & 0 & 1 & 1 & 0 & 0 & 1 & 1 & 0 & 0 & 1 & 1 &  \dots\\ \hline
 3 & 1 & 0 & 0 & 1 & 1 & 3 & 0 & 1 & 1 & 0 & 0 &  \dots \\ \hline
\end{tabular}
}
\end{table}

\begin{table}[H] 
  \centering 
  \caption{Grundy Numbers for $x=12n-4,\dots, 12n+3$}
  \label{table12nm4} 
{
\setlength{\tabcolsep}{1pt} 
\begin{tabular}{@{\hspace{-3.4pt} \vrule width 1.4pt \hspace{0pt}}r|r|r|r @{\hspace{0.4pt} \vrule width 1.4pt \hspace{0pt}}r|r|r|r @{\hspace{0.4pt} \vrule width 1.4pt \hspace{0pt}}} \hline
 $12n{-}4$ & $12n{-}3$ & $12n{-}2$ & $12n{-}1$ &  $12n$ & $12n{+}1$ & $12n{+}2$ & $12n{+}3$  \\ \hline
 0 & 0 & 1 & 1     & 2 & 2 & 3 & 3  \\ \hline
 1 & 1 & 0 & 0     & 1 & 3 & 2 & 2  \\ \hline
\end{tabular}
}
\end{table}

\begin{table}[H] 
  \centering 
  \caption{Grundy Numbers for $x=12n+4,\dots, 12n+8$}
  \label{table12n4} 
{
\setlength{\tabcolsep}{1pt} 
\begin{tabular}{@{\hspace{-3.4pt} \vrule width 1.4pt \hspace{0pt}}r|r|r|r@{\hspace{0.7pt} \vrule width 1.4pt \hspace{0.0pt}}  r |}\hline
$12n{+}4$ & $12n{+}5$ & $12n{+}6$ & $12n{+}7$ &  $12n{+}8$  \\ \hline
 2 & 2 & 3 & 3 & 2    \\ \hline
 3 & 3 & 0 & 2 & 4   \\ \hline
\end{tabular}
}
\end{table}

\begin{definition}\label{defofabcdedfg11}
Let $A=\{0,1,2,0\}$, $B=\{1,1,0,0,\}$, $C=\{1,3,4,2\}$, $D=\{0,3,2,2,\}$, $E=\{3,3,2,2\}$, $F=\{3,1,0,0,1,1,3,0\}$, and $G=\{1,3,2,2,3,3,0,2,4\}$.     
\end{definition}

\begin{example}
For the sets of numbers defined in Definition \ref{defofabcdedfg11}, we denote by 
$AB$ the set of numbers $\{0,1,2,0,1,1,0,0\}$ and by $B^2$ the set of numbers $\{1,1,0,0,1,1,0,0\}$.
\end{example}

\begin{theorem}
For Grundy numbers of the subtraction Nim with the subtraction set $\{s_1,s_2,s_3\}$ $= \{2, 4n, 4n+2\}$, we have the following equations.
\begin{equation}
\{ \mathcal{G}(k,1): k=1,2, \dots, 12n+8 \} = AB^{n-1}CDE^{n-2}FB^{n-2}G.\label{abcdefbcform}
\end{equation}
\end{theorem}
\begin{proof} We have five cases.\\
\noindent {\tt Case 1}: We prove Equation (\ref{abcdefbcform}) for $x=0,1,\dots, 4n+3$, and we use Table \ref{tables4n3}.

By Theorem \ref{lemmafor4mcase}, for $m \leq n-1$ 
and $t=0,1$,
\begin{equation}
\mathcal{G}(4m+t,0)=0  \label{4mt00}
\end{equation}
and for $m \leq n-1$ 
and $t=2,3$,
\begin{equation}
\mathcal{G}(4m+t,0)=1.  \label{4mt01}  
\end{equation}
Since we can use only the subtraction $s_1=2$,
$\mathcal{G}(0,1)=mex(\emptyset )=0$,
$\mathcal{G}(1,1)=mex(\{\mathcal{G}(1,0)\})=mex(\{0\})=1,$
$\mathcal{G}(2,1)=mex(\{\mathcal{G}(0,1), \mathcal{G}(2,0)\})=mex(\{0,1\})=2,$
$\mathcal{G}(3,1)=mex(\{\mathcal{G}(1,1), \mathcal{G}(3,0)\})= $ $mex(\{1,1\})=0,$
$\mathcal{G}(4,1)= $  $mex(\{\mathcal{G}(2,1),$ $ \mathcal{G}(4,0)\}) $ $=mex(\{2,0\})=1,$
$\mathcal{G}(5,1)=$ $mex(\{\mathcal{G}(3,1), \mathcal{G}(5,0)\})=mex(\{0,0\})=1,$   
$\mathcal{G}(6,1)=$ $mex(\{\mathcal{G}(4,1), \mathcal{G}(6,0)\})=mex(\{1,1\})=0,$ 
$\mathcal{G}(7,1)=$ \\ $mex(\{\mathcal{G}(5,1),\mathcal{G}(7,0)\})=mex(\{1,1\})=0.$ 

Next, we prove that 
\begin{equation}
\mathcal{G}(4m+t,1)  = 1  \label{formula4mt12}
\end{equation}
for $t=0,1$ and $2 \leq m \leq n-1$, and 
\begin{equation}
\mathcal{G}(4m+t,1)=0  \label{formula4mt23}
\end{equation}
for $t=2,3$ and $1 \leq m \leq n-1$ by mathematical induction. Here we use only $s_1=2$.
We assume that Equations (\ref{formula4mt12}) and (\ref{formula4mt23}) are 
valid for $m$ such that $m < p$. For $t=0,1$, $t+2=2,3$, and hence, by 
Equations (\ref{4mt00}), (\ref{4mt01}) and 
mathematical induction assumption for Equations (\ref{formula4mt12}) and (\ref{formula4mt23}), 
\begin{equation}
\mathcal{G}(4p+t,1)=mex(\{\mathcal{G}(4(p-1)+t+2,1), \mathcal{G}(4p+t,0)\})=mex(\{0,0\})=1 \nonumber   
\end{equation}
 for $2 \leq p \leq n-1$, and 
\begin{equation}
\mathcal{G}(4p+t,1)= mex(\{\mathcal{G}(4p+t-2,1), \mathcal{G}(4p+t,0)\})=mex(\{1,1\})=0 \nonumber   
\end{equation}
for $t=2,3$ and $1 \leq p \leq n-1$.

For $x=4n, 4n+1$, we use $s_1=2$ and $s_2=4n$.
$\mathcal{G}(4n,1)$ $=mex(\{\mathcal{G}(0,1), \mathcal{G}(4n-2,1), \mathcal{G}(4n,0)\}) $ $=mex(\{0,0,2\})=1$.
$\mathcal{G}(4n+1,1)$ $=mex(\{\mathcal{G}(1,1), \mathcal{G}(4n-1,1), \mathcal{G}(4n+1,0)\}) $ $=mex(\{1,0,2\})=3$.

For $4n+2, 4n+3$, we use $s_1=2$, $s_2=4n$ and $s_3=4n+2$.
$\mathcal{G}(4n+2,1)$ $=mex(\{\mathcal{G}(0,1), \mathcal{G}(2,1), \mathcal{G}(4n,1),  \mathcal{G}(4n+2,0)\}) $ $=mex(\{0,2,1,3\})=4$.
$\mathcal{G}(4n+3,1)$ $=mex(\{\mathcal{G}(1,1),\mathcal{G}(3,1), \mathcal{G}(4n+1,1), \mathcal{G}(4n+3,0)\}) $ $=mex(\{1,0,3,3\})=2$.\\
\noindent {\tt Case 2}: We prove Equation (\ref{abcdefbcform}) for $x=4n+4, \dots, 8n-1$, and we use Table \ref{table4n4}.
\begin{equation}
\mathcal{G}(4n+4,1)=mex(\{\mathcal{G}(2,1),\mathcal{G}(4,1), \mathcal{G}(4n+2,1), \mathcal{G}(4n+4,0)\}) =mex(\{2,1,4,2\})=0, \nonumber
\end{equation}
\begin{equation}
\mathcal{G}(4n+5,1)=mex(\{\mathcal{G}(3,1),\mathcal{G}(5,1), \mathcal{G}(4n+3,1), \mathcal{G}(4n+5,0)\}) =mex(\{0,1,2,2\})=3, \nonumber
\end{equation}
\begin{equation}
\mathcal{G}(4n+6,1)=mex(\{\mathcal{G}(4,1),\mathcal{G}(6,1), \mathcal{G}(4n+4,1), \mathcal{G}(4n+6,0)\}) =mex(\{1,0,0,3\})=2, \nonumber
\end{equation}
\begin{equation}
\mathcal{G}(4n+7,1)=mex(\{\mathcal{G}(5,1),\mathcal{G}(7,1), \mathcal{G}(4n+5,1), \mathcal{G}(4n+7,0)\}) =mex(\{1,0,3,3\})=2, \nonumber
\end{equation}
For $m$ such that $n+2 \leq m \leq 2n-1$,
$1 \leq m-n-1 \leq n-2$, $2 \leq m-n \leq n-1$ and $n+1 \leq m-1 \leq 2n-2$, and hence for $t=0,1$,
\begin{align}
\mathcal{G}(4m+t,1)& =mex(\{\mathcal{G}(4(m-n-1)+t+2,1),\mathcal{G}(4(m-n)+t,1), \nonumber \\
& \ \ \ \ \ \ \ \ \ \ \ \ \ \mathcal{G}(4(m-1)+t+2,1), \mathcal{G}(4m+t,0)\})  \nonumber \\
& = mex(\{0,1,2,2\})=3,\nonumber 
\end{align}
and for $t=2,3$,
\begin{align}
\mathcal{G}(4m+t,1)& =mex(\{\mathcal{G}(4(m-n)+t-2,1),\mathcal{G}(4(m-n)+t,1), \nonumber \\
& \ \ \ \ \ \ \ \ \ \ \ \ \ \mathcal{G}(4m+t-2,1), \mathcal{G}(4m+t,0)\})  \nonumber \\
& = mex(\{1,0,3,3\})=2.\nonumber 
\end{align}
\noindent {\tt Case 3}: We prove Equation (\ref{abcdefbcform}) for $x=8n, \dots, 8n+11$, and we use Table \ref{table8n}.
\begin{align}
\mathcal{G}(8n,1)&=mex(\{\mathcal{G}(4n-2,1),\mathcal{G}(4n,1), \mathcal{G}(8n-2,1), \mathcal{G}(8n,0)\}) \nonumber \\
& \ \ \ \ \ \ \ \ \ \ \ \ \ =mex(\{0,1,2,0\})=3, \nonumber \\
\mathcal{G}(8n+1,1)&=mex(\{\mathcal{G}(4n-1,1),\mathcal{G}(4n+1,1), \mathcal{G}(8n-1,1), \mathcal{G}(8n+1,0)\})\nonumber \\ 
& \ \ \ \ \ \ \ \ \ \ \ \ \ =mex(\{0,3,2,0\})=1, \nonumber
\end{align}
\begin{align}
\mathcal{G}(8n+2,1)& =mex(\{\mathcal{G}(4n,1),\mathcal{G}(4n+2,1), \mathcal{G}(8n,1), \mathcal{G}(8n+2,0)\}) \nonumber \\
& \ \ \ \ \ \ \ \ \ \ \ \ \ =mex(\{1,4,3,1\})=0, \nonumber \\
\mathcal{G}(8n+3,1)& =mex(\{\mathcal{G}(4n+1,1),\mathcal{G}(4n+3,1), \mathcal{G}(8n+1,1), \mathcal{G}(8n+3,0)\})\nonumber \\ 
& \ \ \ \ \ \ \ \ \ \ \ \ \ =mex(\{3,2,1,1\})=0, \nonumber
\end{align}
\begin{align}
\mathcal{G}(8n+4,1)&=mex(\{\mathcal{G}(4n+2,1),\mathcal{G}(4n+4,1), \mathcal{G}(8n+2,1), \mathcal{G}(8n+4,0)\}) \nonumber \\
& \ \ \ \ \ \ \ \ \ \ \ \ \ =mex(\{4,0,0,0\})=1, \nonumber
\end{align}
\begin{align}
\mathcal{G}(8n+5,1)& =mex(\{\mathcal{G}(4n+3,1),\mathcal{G}(4n+5,1), \mathcal{G}(8n+3,1), \mathcal{G}(8n+5,0)\}) \nonumber \\
& \ \ \ \ \ \ \ \ \ \ \ \ \ =mex(\{2,3,0,0\})=1, \nonumber
\end{align}
\begin{align}
\mathcal{G}(8n+6,1)&=mex(\{\mathcal{G}(4n+4,1),\mathcal{G}(4n+6,1), \mathcal{G}(8n+4,1), \mathcal{G}(8n+6,0)\}) \nonumber \\
& \ \ \ \ \ \ \ \ \ \ \ \ \ =mex(\{0,2,1,1\})=3, \nonumber
\end{align}
\begin{align}
\mathcal{G}(8n+7,1)&=mex(\{\mathcal{G}(4n+5,1),\mathcal{G}(4n+7,1), \mathcal{G}(8n+5,1), \mathcal{G}(8n+7,0)\}) \nonumber \\
& \ \ \ \ \ \ \ \ \ \ \ \ \ =mex(\{3,2,1,1\})=0, \nonumber
\end{align}
\begin{align}
\mathcal{G}(8n+8,1)&=mex(\{\mathcal{G}(4n+6,1),\mathcal{G}(4n+8,1), \mathcal{G}(8n+6,1), \mathcal{G}(8n+8,0)\}) \nonumber \\
& \ \ \ \ \ \ \ \ \ \ \ \ \ =mex(\{2,3,3,0\})=1, \nonumber
\end{align}
\begin{align}
\mathcal{G}(8n+9,1)&=mex(\{\mathcal{G}(4n+7,1),\mathcal{G}(4n+9,1), \mathcal{G}(8n+7,1), \mathcal{G}(8n+9,0)\}) \nonumber \\
& \ \ \ \ \ \ \ \ \ \ \ \ \ =mex(\{2,3,0,0\})=1, \nonumber
\end{align}
\begin{align}
\mathcal{G}(8n+10,1)&=mex(\{\mathcal{G}(4n+8,1),\mathcal{G}(4n+10,1), \mathcal{G}(8n+8,1), \mathcal{G}(8n+10,0)\}) \nonumber \\
& \ \ \ \ \ \ \ \ \ \ \ \ \ =mex(\{3,2,1,1\})=0, \nonumber
\end{align}
\begin{align}
\mathcal{G}(8n+11,1)&=mex(\{\mathcal{G}(4n+9,1),\mathcal{G}(4n+11,1), \mathcal{G}(8n+9,1), \mathcal{G}(8n+11,0)\}) \nonumber \\
& \ \ \ \ \ \ \ \ \ \ \ \ \ =mex(\{3,2,1,1\})=0. \nonumber
\end{align}
\noindent {\tt Case 4}: We prove Equation (\ref{abcdefbcform}) for $x=8n+12, \dots, 12n-1$, and we use the last part of Table \ref{table8n} and the first part of Table \ref{table12nm4}.

For $m$ such that $3 \leq m \leq n-1$, we obtain
\begin{equation}
4n+12 \leq 4n+4m \leq 8n-4, \nonumber
\end{equation}
\begin{equation}
4n+8 \leq 4n+4(m-1) \leq 8n-8,    \nonumber
\end{equation}
\begin{equation}
8n+8 \leq 8n+4(m-1) \leq 12n-8,    \nonumber
\end{equation}
and 
\begin{equation}
8n+12 \leq 8n+4m \leq 12n-4.    \nonumber
\end{equation}
Hence for $t=0,1$,
\begin{align}
\mathcal{G}(8n+4m+t,1)& =mex(\{\mathcal{G}(4n+4(m-1)+t+2,1),\mathcal{G}(4n+4m+t,1), \nonumber \\
& \ \ \ \ \ \ \ \ \ \ \ \ \ \mathcal{G}(8n+4(m-1)+t+2,1), \mathcal{G}(8n+4m+t,0)\})  \nonumber \\
& = mex(\{2,3,0,0\})=1,\nonumber 
\end{align}
and for $t=2,3$,
\begin{align}
\mathcal{G}(8n+4m+t,1)& =mex(\{\mathcal{G}(4n+4m+t-2,1),\mathcal{G}(4n+4m+t,1), \nonumber \\
& \ \ \ \ \ \ \ \ \ \ \ \ \ \mathcal{G}(8n+4m+t-2,1), \mathcal{G}(8n+4m+t,0)\})  \nonumber \\
& = mex(\{3,2,1,1\})=0.\nonumber 
\end{align}
\noindent {\tt Case 5}: We prove Equation (\ref{abcdefbcform}) for $x=12n, \dots, 12n+8$, and we use last half of Table \ref{table12n4} and  Table \ref{table12n4}.
 \begin{align}
\mathcal{G}(12n,1)&=mex(\{\mathcal{G}(8n-2,1),\mathcal{G}(8n,1), \mathcal{G}(12n-2,1), \mathcal{G}(12n,0)\}) \nonumber \\
& \ \ \ \ \ \ \ \ \ \ \ \ \ =mex(\{2,3,0,2\})=1, \nonumber \\
\mathcal{G}(12n+1,1)&=mex(\{\mathcal{G}(8n-1,1),\mathcal{G}(8n+1,1), \mathcal{G}(12n-1,1), \mathcal{G}(12n+1,0)\}) \nonumber \\
& \ \ \ \ \ \ \ \ \ \ \ \ \ =mex(\{2,1,0,2\})=3, \nonumber \\
\mathcal{G}(12n+2,1)&=mex(\{\mathcal{G}(8n,1),\mathcal{G}(8n+2,1), \mathcal{G}(12n,1), \mathcal{G}(12n+2,0)\}) \nonumber \\
& \ \ \ \ \ \ \ \ \ \ \ \ \ =mex(\{3,0,1,3\})=2, \nonumber
\end{align}
\begin{align}
\mathcal{G}(12n+3,1)&=mex(\{\mathcal{G}(8n+1,1),\mathcal{G}(8n+3,1), \mathcal{G}(12n+1,1), \mathcal{G}(12n+3,0)\}) \nonumber \\
& \ \ \ \ \ \ \ \ \ \ \ \ \ =mex(\{1,0,3,3\})=2, \nonumber
\end{align}
\begin{align}
\mathcal{G}(12n+4,1)&=mex(\{\mathcal{G}(8n+2,1),\mathcal{G}(8n+4,1), \mathcal{G}(12n+2,1), \mathcal{G}(12n+4,0)\}) \nonumber \\
& \ \ \ \ \ \ \ \ \ \ \ \ \ =mex(\{0,1,2,2\})=3, \nonumber
\end{align}
\begin{align}
\mathcal{G}(12n+5,1)&=mex(\{\mathcal{G}(8n+3,1),\mathcal{G}(8n+5,1), \mathcal{G}(12n+3,1), \mathcal{G}(12n+5,0)\}) \nonumber \\
& \ \ \ \ \ \ \ \ \ \ \ \ \ =mex(\{0,1,2,2\})=3, \nonumber
\end{align}
\begin{align}
\mathcal{G}(12n+6,1)&=mex(\{\mathcal{G}(8n+4,1),\mathcal{G}(8n+6,1), \mathcal{G}(12n+4,1), \mathcal{G}(12n+6,0)\}) \nonumber \\
& \ \ \ \ \ \ \ \ \ \ \ \ \ =mex(\{1,3,3,3\})=0, \nonumber
\end{align}
\begin{align}
\mathcal{G}(12n+7,1)&=mex(\{\mathcal{G}(8n+5,1),\mathcal{G}(8n+7,1), \mathcal{G}(12n+5,1), \mathcal{G}(12n+7,0)\}) \nonumber \\
& \ \ \ \ \ \ \ \ \ \ \ \ \ =mex(\{1,0,3,3\})=2, \nonumber
\end{align}
\begin{align}
\mathcal{G}(12n+8,1)&=mex(\{\mathcal{G}(8n+6,1),\mathcal{G}(8n+8,1), \mathcal{G}(12n+6,1), \mathcal{G}(12n+8,0)\}) \nonumber \\
& \ \ \ \ \ \ \ \ \ \ \ \ \ =mex(\{3,1,0,2\})=4. \nonumber
\end{align}
\end{proof}

\begin{table}[H] 
  \centering 
  \caption{Grundy Numbers for $x=12n+9,\dots, 16n-4$}
  \label{table12n9} 
{
\setlength{\tabcolsep}{1pt} 
\begin{tabular}{@{\hspace{-3.5pt} \vrule width 1.4pt \hspace{3pt}}r|r|r|r@{\hspace{0.4pt} \vrule width 1.4pt \hspace{3pt}}r@{\hspace{0.4pt} \vrule width 1.4pt \hspace{0.0pt}}r|r|r|r @{\hspace{0.4pt} \vrule width 1.4pt \hspace{0pt}}} \hline
 $12n{+}9$ & $12n{+}10$ & $12n{+}11$ & $12n{+}12$ & \dots & $16n{-}7$ & $16n{-}6$ & $16n{-}5$ &  $16n{-}4$  \\ \hline
 2 & 3 & 3 & 2 & \dots & 2 & 3 & 3 & 2 \\ \hline
 3 & 2 & 2 & 3 & \dots & 3 & 2 & 2 & 3 \\ \hline
\end{tabular}
}
\end{table}

\begin{table}[H] 
  \centering 
  \caption{Grundy Numbers for $x=16n-3,\dots, 16n+4$}
  \label{table16nm3} 
{
\setlength{\tabcolsep}{1pt} 
\begin{tabular}{@{\hspace{-3.4pt} \vrule width 1.4pt \hspace{0pt}}r|r|r|r @{\hspace{0.4pt} \vrule width 1.4pt \hspace{0pt}}r|r|r|r @{\hspace{0.4pt} \vrule width 1.4pt \hspace{0pt}}} \hline
 $16n{-}3$ & $16n{-}2$ & $16n{-}1$ & $16n$ &  $16n{+}1$ & $16n{+}2$ & $16n{+}3$ & $16n{+}4$  \\ \hline
 2 & 3 & 3 & 0     & 0 & 1 & 1 & 0  \\ \hline
 3 & 2 & 2 & 3     & 1 & 0 & 0 & 1  \\ \hline
\end{tabular}
}
\end{table}

\begin{table}[H] 
  \centering 
  \caption{Grundy Numbers for $x=16n+5,\dots, 16n+12$}
  \label{table16n5} 
{
\setlength{\tabcolsep}{1pt} 
\begin{tabular}{@{\hspace{-3.4pt} \vrule width 1.4pt \hspace{0pt}}r|r|r|r @{\hspace{0.4pt} \vrule width 1.4pt \hspace{0pt}} r|r|r|r @{\hspace{0.4pt} \vrule width 1.4pt \hspace{0pt}}r |} \hline
 $16n{+}5$ & $16n{+}6$ & $16n{+}7$ & $16n{+}8$ & $16n{+}9$ & $16n{+}10$ & $16n{+}11$ &   $16n{+}12$ & \dots  \\ \hline
 0 & 1 & 1 & 0    & 0 & 1 & 1 & 0 & \dots \\ \hline
1 & 2 & 0 & 1    & 1 & 0 & 0 & 1 & \dots \\ \hline
\end{tabular}
}
\end{table}

\begin{table}[H] 
  \centering 
  \caption{Grundy Numbers for $x=20n-7,\dots, 20n$}
  \label{table20n7} 
{
\setlength{\tabcolsep}{1pt} 
\begin{tabular}{@{\hspace{-3.4pt} \vrule width 1.4pt \hspace{0pt}}r|r|r|r @{\hspace{0.4pt} \vrule width 1.4pt \hspace{0pt}}r|r|r|r@{\hspace{0.4pt} \vrule width 1.4pt \hspace{0pt}}} \hline
 $20n{-}7$ & $20n{-}6$ & $20n{-}5$ & $20n{-}4$ & $20n{-}3$ & $20n{-}2$ & $20n{-}1$ &   $20n$   \\ \hline
 0 & 1 & 1 & 0    & 0 & 1 & 1 & 2 \\ \hline
1 & 0 & 0 & 1    & 1 & 0 & 0 & 1  \\ \hline
\end{tabular}
}
\end{table}

\begin{table}[H] 
  \centering 
  \caption{Grundy Numbers for $x=20n+1,\dots, 20n+8$}
  \label{table20n1} 
{
\setlength{\tabcolsep}{1pt} 
\begin{tabular}{@{\hspace{-3.4pt} \vrule width 1.4pt \hspace{0pt}}r|r|r|r@{\hspace{0.4pt} \vrule width 1.4pt \hspace{0pt}}r|r|r|r@{\hspace{0.4pt} \vrule width 1.4pt \hspace{0pt}}} \hline
 $20n{+}1$ & $20n{+}2$ & $20n{+}3$ & $20n{+}4$ & $20n{+}5$ & $20n{+}6$ & $20n{+}7$ &   $20n{+}8$   \\ \hline
2 & 3 & 3 & 2    & 2 & 3 & 3 & 2 \\ \hline
3 & 2 & 2 & 3    & 3 & 0 & 2 & 3  \\ \hline
\end{tabular}
}
\end{table}

\begin{definition}\label{defofabcdedfg}
Let $P=\{3,2,2,3\}$, $Q=\{1,0,0,1\}$, $R=\{1,2,0,1\}$, $S=\{3,0,2,3\}$.     
\end{definition}

Theorem \ref{periodicth} presents the periodic part of Grundy numbers $\{ \mathcal{G}(x,1): x=0,1,2,\dots \}$.

\begin{theorem}\label{periodicth}
For Grundy numbers of the subtraction Nim with the subtraction set $\{s_1,s_2,s_3\}$ $= \{2, 4n, 4n+2\}$ and a pass, we have the following equation.
\begin{equation}
\{ \mathcal{G}(x,1): x=12n+9, \dots 20n+8  \} = P^{n-2}QRQ^{n-2}PS.\label{periodpattern}
\end{equation}
\end{theorem}
\begin{proof}
\noindent {\tt Case 1}: We prove Equation (\ref{periodpattern}) for $x=12n+9, \dots, 12n+12$, and we use Table \ref{table12n9}.
\begin{align}
\mathcal{G}(12n+9,1)&=mex(\{\mathcal{G}(8n+7,1),\mathcal{G}(8n+9,1), \mathcal{G}(12n+7,1), \mathcal{G}(12n+9,0)\}) \nonumber \\
& =mex(\{0,1,2,2\})=3, \nonumber \\
\mathcal{G}(12n+10,1)&=mex(\{\mathcal{G}(8n+8,1),\mathcal{G}(8n+10,1), \mathcal{G}(12n+8,1), \mathcal{G}(12n+10,0)\}) \nonumber \\
&  =mex(\{1,0,4,3\})=2, \label{first4} \\
\mathcal{G}(12n+11,1)&=mex(\{\mathcal{G}(8n+9,1),\mathcal{G}(8n+11,1), \mathcal{G}(12n+9,1), \mathcal{G}(12n+11,0)\}) \nonumber \\
&  =mex(\{1,0,3,3\})=2, \nonumber \\
\mathcal{G}(12n+12,1)&=mex(\{\mathcal{G}(8n+10,1),\mathcal{G}(8n+12,1), \mathcal{G}(12n+10,1), \mathcal{G}(12n+12,0)\}) \nonumber \\
&  =mex(\{0,1,2,2\})=3. \nonumber
\end{align}
\noindent {\tt Case 2}: We prove Equation (\ref{periodpattern}) for $x=12n+4m+t$ for $3 \leq m \leq n-2$ and $t=1,2,3,4$ using mathematical induction. Here, $12 \leq 4m \leq 4n-8$ and $8 \leq 4(m-1) \leq 4n-12$.
Here, we use Table \ref{table12n9}.
\begin{align}
\mathcal{G}(12n+4m+1,1)& =mex(\{\mathcal{G}(8n+4(m-1)+3,1),\mathcal{G}(8n+4m+1,1), \nonumber \\
& \ \ \ \ \ \ \ \ \ \ \ \ \ \mathcal{G}(12n+4(m-1)+3,1), \mathcal{G}(12n+4m+1,0)\})  \nonumber \\
& = mex(\{0,1,2,2\})=3,\nonumber \\
\mathcal{G}(12n+4m+2,1)& =mex(\{\mathcal{G}(8n+4m,1),\mathcal{G}(8n+4m+2,1), \nonumber \\
& \ \ \ \ \ \ \ \ \ \ \ \ \ \mathcal{G}(12n+4m,1), \mathcal{G}(12n+4m+2,0)\})  \nonumber \\
& = mex(\{1,0,3,3\})=2,\nonumber \\
\mathcal{G}(12n+4m+3,1)& =mex(\{\mathcal{G}(8n+4m+1,1),\mathcal{G}(8n+4m+3,1), \nonumber \\
& \ \ \ \ \ \ \ \ \ \ \ \ \ \mathcal{G}(12n+4m+1,1), \mathcal{G}(12n+4m+3,0)\})  \nonumber \\
& = mex(\{1,0,3,3\})=2,\nonumber \\
\mathcal{G}(12n+4m+4,1)& =mex(\{\mathcal{G}(8n+4m+2,1),\mathcal{G}(8n+4m+4,1), \nonumber \\
& \ \ \ \ \ \ \ \ \ \ \ \ \ \mathcal{G}(12n+4m+2,1), \mathcal{G}(12n+4m+4,0)\})  \nonumber \\
& = mex(\{0,1,2,2\})=3.\nonumber 
\end{align}
\noindent {\tt Case 3}: We prove Equation (\ref{periodpattern}) for $x=16n-3 \dots, 16n+4$. Here, we use Table \ref{table16nm3}.
\begin{align}
\mathcal{G}(16n-3,1)&=mex(\{\mathcal{G}(12n-5,1),\mathcal{G}(12n-3,1), \mathcal{G}(16n-5,1), \mathcal{G}(16n-3,0)\}) \nonumber \\
&  =mex(\{0,1,2,2\})=3, \nonumber \\
\mathcal{G}(16n-2,1)&=mex(\{\mathcal{G}(12n-4,1),\mathcal{G}(12n-2,1), \mathcal{G}(16n-4,1), \mathcal{G}(16n-2,0)\}) \nonumber \\
&  =mex(\{1,0,3,3\})=2, \nonumber
\end{align}
\begin{align}
\mathcal{G}(16n-1,1)&=mex(\{\mathcal{G}(12n-3,1),\mathcal{G}(12n-1,1), \mathcal{G}(16n-3,1), \mathcal{G}(16n-1,0)\}) \nonumber \\
&  =mex(\{1,0,3,3\})=2, \nonumber \\
\mathcal{G}(16n,1)&=mex(\{\mathcal{G}(12n-2,1),\mathcal{G}(12n,1), \mathcal{G}(16n-2,1), \mathcal{G}(16n,0)\}) \nonumber \\
&  =mex(\{0,1,2,0\})=3, \nonumber
\end{align}
\begin{align}
\mathcal{G}(16n+1,1)&=mex(\{\mathcal{G}(12n-1,1),\mathcal{G}(12n+1,1), \mathcal{G}(16n-1,1), \mathcal{G}(16n+1,0)\}) \nonumber \\
&  =mex(\{0,3,2,0\})=1, \nonumber \\
\mathcal{G}(16n+2,1)&=mex(\{\mathcal{G}(12n,1),\mathcal{G}(12n+2,1), \mathcal{G}(16n,1), \mathcal{G}(16n+2,0)\}) \nonumber \\
&  =mex(\{1,2,3,1\})=0, \nonumber
\end{align}
\begin{align}
\mathcal{G}(16n+3,1)&=mex(\{\mathcal{G}(12n+1,1),\mathcal{G}(12n+3,1), \mathcal{G}(16n+1,1), \mathcal{G}(16n+3,0)\}) \nonumber \\
&  =mex(\{3,2,1,1\})=0, \nonumber
\end{align}
\begin{align}
\mathcal{G}(16n+4,1)&=mex(\{\mathcal{G}(12n+2,1),\mathcal{G}(12n+4,1), \mathcal{G}(16n+2,1), \mathcal{G}(16n+4,0)\}) \nonumber \\
&  =mex(\{2,3,0,0\})=1. \nonumber
\end{align}
\noindent {\tt Case 4}: We prove Equation (\ref{periodpattern}) for $x=16n+5 \dots, 16n+12$. Here, we use Table \ref{table16n5}.
\begin{align}
\mathcal{G}(16n+5,1)&=mex(\{\mathcal{G}(12n+3,1),\mathcal{G}(12n+5,1), \mathcal{G}(16n+3,1), \mathcal{G}(16n+5,0)\}) \nonumber \\
&  =mex(\{2,3,0,0\})=1, \nonumber
\end{align}
\begin{align}
\mathcal{G}(16n+6,1)&=mex(\{\mathcal{G}(12n+4,1),\mathcal{G}(12n+6,1), \mathcal{G}(16n+4,1), \mathcal{G}(16n+6,0)\}) \nonumber \\
&  =mex(\{3,0,1,1\})=2, \nonumber
\end{align}
\begin{align}
\mathcal{G}(16n+7,1)&=mex(\{\mathcal{G}(12n+5,1),\mathcal{G}(12n+7,1), \mathcal{G}(16n+5,1), \mathcal{G}(16n+7,0)\}) \nonumber \\
&  =mex(\{3,2,1,1\})=0, \nonumber
\end{align}
\begin{align}
\mathcal{G}(16n+8,1)&=mex(\{\mathcal{G}(12n+6,1),\mathcal{G}(12n+8,1), \mathcal{G}(16n+6,1), \mathcal{G}(16n+8,0)\}) \nonumber \\
&  =mex(\{0,4,2,0\})=1, \label{second4}     
\end{align}
\begin{align}
\mathcal{G}(16n+9,1)&=mex(\{\mathcal{G}(12n+7,1),\mathcal{G}(12n+9,1), \mathcal{G}(16n+7,1), \mathcal{G}(16n+9,0)\}) \nonumber \\
&  =mex(\{2,3,0,0\})=1, \nonumber
\end{align}
\begin{align}
\mathcal{G}(16n+10,1)&=mex(\{\mathcal{G}(12n+8,1),\mathcal{G}(12n+10,1), \nonumber \\
& \ \ \ \ \ \ \ \ \ \ \ \ \ \ \mathcal{G}(16n+8,1), \mathcal{G}(16n+10,0)\}) \nonumber \\
&  =mex(\{4,2,1,1\})=0, \label{third4} 
\end{align}
\begin{align}
\mathcal{G}(16n+11,1)&=mex(\{\mathcal{G}(12n+9,1),\mathcal{G}(12n+11,1), \nonumber \\
& \ \ \ \ \ \ \ \ \ \ \ \ \ \ \mathcal{G}(16n+9,1), \mathcal{G}(16n+11,0)\}) \nonumber \\
&  =mex(\{3,2,1,1\})=0,  \nonumber  
\end{align}
\begin{align}
\mathcal{G}(16n+12,1)&=mex(\{\mathcal{G}(12n+10,1),\mathcal{G}(12n+12,1), \nonumber \\
& \ \ \ \ \ \ \ \ \ \ \ \ \ \ \mathcal{G}(16n+10,1), \mathcal{G}(16n+12,0)\}) \nonumber \\
&  =mex(\{2,3,0,0\})=1. \nonumber
\end{align}
\noindent {\tt Case 5}: We prove Equation (\ref{periodpattern}) for $x=16n+4m+t$ for $3 \leq m \leq n-2$ and $t=1,2,3,4$ using mathematical induction.
Here, we have $12 \leq 4m \leq 4n-8$ and $8 \leq 4(m-1) \leq 4n-12$.
We use the last part of Table \ref{table16n5} and the first half of  Table \ref{table20n7}.
\begin{align}
\mathcal{G}(16n+4m+1,1)& =mex(\{\mathcal{G}(12n+4(m-1)+3,1),\mathcal{G}(12n+4m+1,1), \nonumber \\
& \ \ \ \ \ \ \ \ \ \ \ \ \ \mathcal{G}(16n+4(m-1)+3,1), \mathcal{G}(16n+4m+1,0)\})  \nonumber \\
& = mex(\{2,3,0,0\})=1,\nonumber 
\end{align}
\begin{align}
\mathcal{G}(16n+4m+2,1)& =mex(\{\mathcal{G}(12n+4m,1),\mathcal{G}(12n+4m+2,1), \nonumber \\
& \ \ \ \ \ \ \ \ \ \ \ \ \ \mathcal{G}(16n+4m,1), \mathcal{G}(16n+4m+2,0)\})  \nonumber \\
& = mex(\{3,2,1,1\})=0,\nonumber 
\end{align}
\begin{align}
\mathcal{G}(16n+4m+3,1)& =mex(\{\mathcal{G}(12n+4m+1,1),\mathcal{G}(12n+4m+3,1), \nonumber \\
& \ \ \ \ \ \ \ \ \ \ \ \ \ \mathcal{G}(16n+4m+1,1), \mathcal{G}(16n+4m+3,0)\})  \nonumber \\
& = mex(\{3,2,1,1\})=0,\nonumber 
\end{align}
\begin{align}
\mathcal{G}(16n+4m+4,1)& =mex(\{\mathcal{G}(12n+4m+2,1),\mathcal{G}(12n+4m+4,1), \nonumber \\
& \ \ \ \ \ \ \ \ \ \ \ \ \ \mathcal{G}(16n+4m+2,1), \mathcal{G}(16n+4m+4,0)\})  \nonumber \\
& = mex(\{2,3,0,0\})=1.\nonumber 
\end{align}
\noindent {\tt Case 6}: We prove Equation (\ref{periodpattern}) for $x=20n-3, 20n-2, \dots, 20n+8$. Here, we use the last half of  Table \ref{table20n7}.

\begin{align}
\mathcal{G}(20n-3,1)&=mex(\{\mathcal{G}(16n-5,1),\mathcal{G}(16n-3,1), \mathcal{G}(20n-5,1), \mathcal{G}(20n-3,0)\}) \nonumber \\
&  =mex(\{2,3,0,0\})=1,. \nonumber
\end{align}
\begin{align}
\mathcal{G}(20n-2,1)&=mex(\{\mathcal{G}(16n-4,1),\mathcal{G}(16n-2,1), \mathcal{G}(20n-4,1), \mathcal{G}(20n-2,0)\}) \nonumber \\
&  =mex(\{3,2,1,1\})=0, \nonumber
\end{align}
\begin{align}
\mathcal{G}(20n-1,1)&=mex(\{\mathcal{G}(16n-3,1),\mathcal{G}(16n-1,1), \mathcal{G}(20n-3,1), \mathcal{G}(20n-1,0)\}) \nonumber \\
&  =mex(\{3,2,1,1\})=0, \nonumber \\
\mathcal{G}(20n,1)&=mex(\{\mathcal{G}(16n-2,1),\mathcal{G}(16n,1), \mathcal{G}(20n-2,1), \mathcal{G}(20n,0)\}) \nonumber \\
&  =mex(\{2,3,0,2\})=1, \nonumber
\end{align}
\begin{align}
\mathcal{G}(20n+1,1)&=mex(\{\mathcal{G}(16n-1,1),\mathcal{G}(16n+1,1), \mathcal{G}(20n-1,1), \mathcal{G}(20n+1,0)\}) \nonumber \\
&  =mex(\{2,1,0,2\})=3, \nonumber \\
\mathcal{G}(20n+2,1)&=mex(\{\mathcal{G}(16n,1),\mathcal{G}(16n+2,1), \mathcal{G}(20n,1), \mathcal{G}(20n+2,0)\}) \nonumber \\
&  =mex(\{3,0,1,3\})=2, \nonumber
\end{align}
\begin{align}
\mathcal{G}(20n+3,1)&=mex(\{\mathcal{G}(16n+1,1),\mathcal{G}(16n+3,1), \mathcal{G}(20n+1,1), \mathcal{G}(20n+3,0)\}) \nonumber \\
&  =mex(\{1,0,3,3\})=2, \nonumber
\end{align}
\begin{align}
\mathcal{G}(20n+4,1)&=mex(\{\mathcal{G}(16n+2,1),\mathcal{G}(16n+4,1), \mathcal{G}(20n+2,1), \mathcal{G}(20n+4,0)\}) \nonumber \\
&  =mex(\{0,1,2,2\})=3. \nonumber
\end{align}
\begin{align}
\mathcal{G}(20n+5,1)&=mex(\{\mathcal{G}(16n+3,1),\mathcal{G}(16n+5,1), \mathcal{G}(20n+3,1), \mathcal{G}(20n+5,0)\}) \nonumber \\
&  =mex(\{0,1,2,2\})=3, \nonumber
\end{align}
\begin{align}
\mathcal{G}(20n+6,1)&=mex(\{\mathcal{G}(16n+4,1),\mathcal{G}(16n+6,1), \mathcal{G}(20n+4,1), \mathcal{G}(20n+6,0)\}) \nonumber \\
&  =mex(\{1,2,3,3\})=0, \nonumber
\end{align}
\begin{align}
\mathcal{G}(20n+7,1)&=mex(\{\mathcal{G}(16n+5,1),\mathcal{G}(16n+7,1), \mathcal{G}(20n+5,1), \mathcal{G}(20n+7,0)\}) \nonumber \\
&  =mex(\{1,0,3,3\})=2, \nonumber
\end{align}
\begin{align}
\mathcal{G}(20n+8,1)&=mex(\{\mathcal{G}(16n+6,1),\mathcal{G}(16n+8,1), \mathcal{G}(20n+6,1), \mathcal{G}(20n+8,0)\}) \nonumber \\
&  =mex(\{2,1,0,2\})=3, \nonumber
\end{align}
\end{proof}

\begin{theorem}
From $x=12n+9$, the sequence of Grundy numbers enters into a loop, i.e., for  $t \in \mathbb{Z}_{\ge 0}$ it
satisfies the equation 
\begin{equation}
\{\mathcal{G}(x,1):x=12n+9+8nt,12n+10+8nt, \dots, 20n+8+8nt\}=P^{n-2}QRQ^{n-2}PS. \label{12n912n10} 
\end{equation}
\end{theorem}
\begin{proof}
The last $4n+2$ numbers of $AB^{n-1}CDE^{n-2}FB^{n-2}G$ are
\begin{align}
& \{\mathcal{G}(x,1):x=8n+7, \dots, 12n+8\} \nonumber \\
& =\{0\} \cup \{1,1,0,0,\}^{n-2} \cup \{1,3,2,2,3,3,0,2,4\}, \label{abn1cde}
\end{align}
 and 
the last $4n+2$ numbers of $P^{n-2}QRQ^{n-2}PS$ are
\begin{align}
& \{\mathcal{G}(x,1):x=16n+7,16n+8, \dots, 20n+8\}\nonumber \\
& =\{0,1\} \cup \{1,0,0,1\}^{n-2} \cup \{3,2,2,3\} \cup \{3,0,2,3\}. \label{pn2qrq} 
\end{align}
Since $8n$ is the length of the loop for $\{\mathcal{G}(x,0):x=0,1,2, \dots\}$, 
\begin{equation}
\{\mathcal{G}(x,0):x=8n+7, \dots, 12n+8\}=\{\mathcal{G}(x,0):x=16n+7, \dots, 20n+8\}. \label{samefornopas}  
\end{equation}

In Equations (\ref{pn2qrq}) and (\ref{abn1cde}), the only difference is 
\begin{equation}
\mathcal{G}(12n+8,1)=4 \text{ and } \mathcal{G}(20n+8,1)=3.  \label{dif34} 
\end{equation}
In Equations (\ref{first4}), (\ref{second4}), and (\ref{third4}), 
we use $mex(\{1,0,4,3\})=2$, $mex(\{0,4,2,0\})=1$, and $mex(\{4,2,1,1\})=0$, and
if we use $3$ instead of $4$, we obtain 
$mex(\{1,0,3,3\})=2$, $mex(\{0,3,2,0\})=1$, and $mex(\{3,2,1,1\})=0$.
Therefore, by Equations (\ref{abn1cde}), (\ref{pn2qrq}), (\ref{samefornopas}), and (\ref{dif34}), we obtain Equation (\ref{12n912n10}).
\end{proof}

When  $x\geq 12n+9$, we have the reverse-$mex$ property of Grundy numbers with a pass.

\begin{theorem}
When  $x\geq 12n+9$, the sequence $\mathcal{G}(x,1)$ satisfies the following equation.
\begin{equation}
\mathcal{G}(x,1)=mex(\{\mathcal{G}(x+4n+2,1),\mathcal{G}(x+4n,1),\mathcal{G}(x+2,1)\}). \label{inversemex}
\end{equation}
\end{theorem}
\begin{proof}
\noindent {\tt Case 1}: We prove Equation (\ref{inversemex}) for $x=12n+4m+t$ for $2 \leq m \leq n-2$ and $t=1,2,3,4$. 
Here, $12n+9 \leq x \leq 16n-4$, $16n+9 \leq 16n+4m+1 < 16n+4m+6 \leq 16n+14$, and 
$12n+11 \leq 12n+4m+3 < 12n+4m+6 \leq 12n+14$.
\begin{align}
\mathcal{G}(12n+4m+1,1)& =mex(\{\mathcal{G}(16n+4m+3,1),\mathcal{G}(16n+4m+1,1), \nonumber \\
& \ \ \ \ \ \mathcal{G}(12n+4m+3,1))\}) = mex(\{0,1,2\})=3,\nonumber 
\end{align}
\begin{align}
\mathcal{G}(12n+4m+2,1)& =mex(\{\mathcal{G}(16n+4m+4,1),\mathcal{G}(16n+4m+2,1), \nonumber \\
& \ \ \ \ \ \mathcal{G}(12n+4m+4,1))\}) = mex(\{1,0,3\})=2,\nonumber 
\end{align}
\begin{align}
\mathcal{G}(12n+4m+3,1)& =mex(\{\mathcal{G}(16n+4m+5,1),\mathcal{G}(16n+4m+3,1), \nonumber \\
& \ \ \ \ \ \mathcal{G}(12n+4m+5,1))\}) = mex(\{1,0,3\})=2,\nonumber 
\end{align}
\begin{align}
\mathcal{G}(12n+4m+4,1)& =mex(\{\mathcal{G}(16n+4m+6,1),\mathcal{G}(16n+4m+4,1), \nonumber \\
& \ \ \ \ \ \mathcal{G}(12n+4m+6,1))\}) = mex(\{0,1,2\})=3,\nonumber 
\end{align}
\noindent {\tt Case 2}: We prove Equation (\ref{inversemex}) for $x=16n-3 \dots, 16n$. 
\begin{align}
\mathcal{G}(16n-3,1)&=mex(\{\mathcal{G}(20n-1,1),\mathcal{G}(20n-3,1), \mathcal{G}(16n-1,1)\}) \nonumber \\
&  =mex(\{0,1,2\})=3, \nonumber \\
\mathcal{G}(16n-2,1)&=mex(\{\mathcal{G}(20n,1),\mathcal{G}(20n-2,1), \mathcal{G}(16n,1)\}) \nonumber \\
&  =mex(\{1,0,3\})=2, \nonumber
\end{align}
\begin{align}
\mathcal{G}(16n-1,1)&=mex(\{\mathcal{G}(20n+1,1),\mathcal{G}(20n-1,1), \mathcal{G}(16n+1,1)\}) \nonumber \\
&  =mex(\{3,0,1\})=2, \nonumber \\
\mathcal{G}(16n,1)&=mex(\{\mathcal{G}(20n+2,1),\mathcal{G}(20n,1), \mathcal{G}(16n+2,1)\}) \nonumber \\
&  =mex(\{2,1,0\})=3. \nonumber
\end{align}
\noindent {\tt Case 3}: We prove Equation (\ref{inversemex}) for $x=16n+1 \dots, 16n+8$. 
\begin{align}
\mathcal{G}(16n+1,1)&=mex(\{\mathcal{G}(20n+3,1),\mathcal{G}(20n+1,1), \mathcal{G}(16n+3,1)\}) \nonumber \\
&  =mex(\{2,3,0\})=1, \nonumber
\end{align}
\begin{align}
\mathcal{G}(16n+2,1)&=mex(\{\mathcal{G}(20n+4,1),\mathcal{G}(20n+2,1), \mathcal{G}(16n+4,1)\}) \nonumber \\
&  =mex(\{3,2,1\})=0, \nonumber
\end{align}
\begin{align}
\mathcal{G}(16n+3,1)&=mex(\{\mathcal{G}(20n+5,1),\mathcal{G}(20n+3,1), \mathcal{G}(16n+5,1)\}) \nonumber \\
&  =mex(\{3,2,1\})=0, \nonumber
\end{align}
\begin{align}
\mathcal{G}(16n+4,1)&=mex(\{\mathcal{G}(20n+6,1),\mathcal{G}(20n+4,1), \mathcal{G}(16n+6,1)\}) \nonumber \\
&  =mex(\{0,3,2\})=1, \nonumber
\end{align}
\begin{align}
\mathcal{G}(16n+5,1)&=mex(\{\mathcal{G}(20n+7,1),\mathcal{G}(20n+5,1), \mathcal{G}(16n+7,1)\}) \nonumber \\
&  =mex(\{2,3,0\})=1, \nonumber
\end{align}
\begin{align}
\mathcal{G}(16n+6,1)&=mex(\{\mathcal{G}(20n+8,1),\mathcal{G}(20n+6,1), \mathcal{G}(16n+8,1)\}) \nonumber \\
&  =mex(\{3,0,1\})=2. \nonumber
\end{align}
Since $P^{n-2}QRQ^{n-2}PS$ is a loop, $\mathcal{G}(20n+9,1)=\mathcal{G}(12n+9,1)=3$. Hence,
\begin{align}
\mathcal{G}(16n+7,1)&=mex(\{\mathcal{G}(20n+9,1),\mathcal{G}(20n+7,1), \mathcal{G}(16n+9,1)\}) \nonumber \\
&  =mex(\{3,2,1\})=0. \nonumber
\end{align}
Similarly, $\mathcal{G}(20n+10,1)=\mathcal{G}(12n+10,1)=2$. Hence, 
\begin{align}
\mathcal{G}(16n+8,1)&=mex(\{\mathcal{G}(20n+10,1),\mathcal{G}(20n+8,1), \mathcal{G}(16n+10,1)\}) \nonumber \\
&  =mex(\{2,3,0\})=1, \nonumber
\end{align}
\noindent {\tt Case 4}: We prove Equation (\ref{inversemex}) for $x=16n+4m+t$ for $2 \leq m \leq n-2$ and $t=1,2,3,4$. Here, $16n+9 \leq x \leq 20n-4$, $12n+9 \leq 12n+4m+1 $
$< 12n+4m+6$ $\leq 16n-4$, and $16n+11 \leq $ $16n+4m+3 < 16n+4m+6$ $\leq 20n-2$.

Note that for $x \geq 20n+9$, $\mathcal{G}(x,1)=\mathcal{G}(x-8n,1)$.
\begin{align}
\mathcal{G}(16n+4m+1,1)& =mex(\{\mathcal{G}(20n+4m+3,1),\mathcal{G}(20n+4m+1,1), \nonumber \\
& \ \ \ \ \ \mathcal{G}(16n+4m+3,1))\}) \nonumber \\
& =mex(\{\mathcal{G}(12n+4m+3,1),\mathcal{G}(12n+4m+1,1), \nonumber \\
& \ \ \ \ \ \mathcal{G}(16n+4m+3,1))\}) \nonumber \\
& = mex(\{2,3,0\})=1,\nonumber 
\end{align}
\begin{align}
\mathcal{G}(16n+4m+2,1)& =mex(\{\mathcal{G}(20n+4m+4,1),\mathcal{G}(20n+4m+2,1), \nonumber \\
& \ \ \ \ \ \mathcal{G}(16n+4m+4,1))\}) \nonumber \\
& =mex(\{\mathcal{G}(12n+4m+4,1),\mathcal{G}(12n+4m+2,1), \nonumber \\
& \ \ \ \ \ \mathcal{G}(16n+4m+4,1))\}) \nonumber \\
&= mex(\{3,2,1\})=0,\nonumber 
\end{align}
\begin{align}
\mathcal{G}(16n+4m+3,1)& =mex(\{\mathcal{G}(20n+4m+5,1),\mathcal{G}(20n+4m+3,1), \nonumber \\
& \ \ \ \ \ \mathcal{G}(16n+4m+5,1))\}) \nonumber \\
& =mex(\{\mathcal{G}(12n+4m+5,1),\mathcal{G}(12n+4m+3,1), \nonumber \\
& \ \ \ \ \ \mathcal{G}(16n+4m+5,1))\}) \nonumber \\
&= mex(\{3,2,1\})=0,\nonumber 
\end{align}
\begin{align}
\mathcal{G}(16n+4m+4,1)& =mex(\{\mathcal{G}(20n+4m+6,1),\mathcal{G}(20n+4m+4,1), \nonumber \\
& \ \ \ \ \ \mathcal{G}(16n+4m+6,1))\}) \nonumber \\
& =mex(\{\mathcal{G}(12n+4m+6,1),\mathcal{G}(12n+4m+4,1), \nonumber \\
& \ \ \ \ \ \mathcal{G}(16n+4m+6,1))\}) \nonumber \\
&= mex(\{2,3,0\})=1.\nonumber 
\end{align}

\noindent {\tt Case 5}: We prove Equation (\ref{inversemex}) for $x=20n-3 \dots, 20n+8$. Note that for $x \geq 20n+9$, $\mathcal{G}(x,1)=\mathcal{G}(x-8n,1)$.
\begin{align}
\mathcal{G}(20n-3,1)&=mex(\{\mathcal{G}(24n-1,1),\mathcal{G}(24n-3,1), \mathcal{G}(20n-1,1)\}) \nonumber \\
&=mex(\{\mathcal{G}(16n-1,1),\mathcal{G}(16n-3,1), \mathcal{G}(20n-1,1)\}) \nonumber \\
&  =mex(\{2,3,0\})=1, \nonumber \\
\mathcal{G}(20n-2,1)&=mex(\{\mathcal{G}(24n,1),\mathcal{G}(24n-2,1), \mathcal{G}(20n,1)\}) \nonumber \\
&=mex(\{\mathcal{G}(16n,1),\mathcal{G}(16n-2,1), \mathcal{G}(20n,1)\}) \nonumber \\
&  =mex(\{3,2,1\})=0, \nonumber
\end{align}
\begin{align}
\mathcal{G}(20n-1,1)&=mex(\{\mathcal{G}(24n+1,1),\mathcal{G}(24n-1,1), \mathcal{G}(20n+1,1)\}) \nonumber \\
&=mex(\{\mathcal{G}(16n+1,1),\mathcal{G}(16n-1,1), \mathcal{G}(20n+1,1)\}) \nonumber \\
&  =mex(\{1,2,3\})=0, \nonumber \\
\mathcal{G}(20n,1)&=mex(\{\mathcal{G}(24n+2,1),\mathcal{G}(24n,1), \mathcal{G}(20n+2,1)\}) \nonumber \\
&=mex(\{\mathcal{G}(16n+2,1),\mathcal{G}(16n,1), \mathcal{G}(20n+2,1)\}) \nonumber \\
&  =mex(\{0,3,2\})=1. \nonumber
\end{align}
\begin{align}
\mathcal{G}(20n+1,1)&=mex(\{\mathcal{G}(24n+3,1),\mathcal{G}(24n+1,1), \mathcal{G}(20n+3,1)\}) \nonumber \\
&=mex(\{\mathcal{G}(16n+3,1),\mathcal{G}(16n+1,1), \mathcal{G}(20n+3,1)\}) \nonumber \\
&  =mex(\{0,1,2\})=3, \nonumber
\end{align}
\begin{align}
\mathcal{G}(20n+2,1)&=mex(\{\mathcal{G}(24n+4,1),\mathcal{G}(24n+2,1), \mathcal{G}(20n+4,1)\}) \nonumber \\
&=mex(\{\mathcal{G}(16n+4,1),\mathcal{G}(16n+2,1), \mathcal{G}(20n+4,1)\}) \nonumber \\
&  =mex(\{1,0,3\})=2, \nonumber
\end{align}
\begin{align}
\mathcal{G}(20n+3,1)&=mex(\{\mathcal{G}(24n+5,1),\mathcal{G}(24n+3,1), \mathcal{G}(20n+5,1)\}) \nonumber \\
&=mex(\{\mathcal{G}(16n+5,1),\mathcal{G}(16n+3,1), \mathcal{G}(20n+5,1)\}) \nonumber \\
&  =mex(\{1,0,3\})=2, \nonumber
\end{align}
\begin{align}
\mathcal{G}(20n+4,1)&=mex(\{\mathcal{G}(24n+6,1),\mathcal{G}(24n+4,1), \mathcal{G}(20n+6,1)\}) \nonumber \\
&=mex(\{\mathcal{G}(16n+6,1),\mathcal{G}(16n+4,1), \mathcal{G}(20n+6,1)\}) \nonumber \\
&  =mex(\{2,1,0\})=3, \nonumber
\end{align}
\begin{align}
\mathcal{G}(20n+5,1)&=mex(\{\mathcal{G}(24n+7,1),\mathcal{G}(24n+5,1), \mathcal{G}(20n+7,1)\}) \nonumber \\
&=mex(\{\mathcal{G}(16n+7,1),\mathcal{G}(16n+5,1), \mathcal{G}(20n+7,1)\}) \nonumber \\
&  =mex(\{0,1,2\})=3, \nonumber
\end{align}
\begin{align}
\mathcal{G}(20n+6,1)&=mex(\{\mathcal{G}(24n+8,1),\mathcal{G}(24n+6,1), \mathcal{G}(20n+8,1)\}) \nonumber \\
&=mex(\{\mathcal{G}(16n+8,1),\mathcal{G}(16n+6,1), \mathcal{G}(20n+8,1)\}) \nonumber \\
&  =mex(\{1,2,3\})=0, \nonumber
\end{align}
\begin{align}
\mathcal{G}(20n+7,1)&=mex(\{\mathcal{G}(24n+9,1),\mathcal{G}(24n+7,1), \mathcal{G}(20n+9,1)\}) \nonumber \\
&=mex(\{\mathcal{G}(16n+9,1),\mathcal{G}(16n+7,1), \mathcal{G}(12n+9,1)\}) \nonumber \\
&  =mex(\{1,0,3\})=2, \nonumber \\
\mathcal{G}(20n+8,1)&=mex(\{\mathcal{G}(24n+10,1),\mathcal{G}(24n+8,1), \mathcal{G}(20n+10,1)\}) \nonumber \\
&=mex(\{\mathcal{G}(16n+10,1),\mathcal{G}(16n+8,1), \mathcal{G}(12n+10,1)\}) \nonumber \\
&  =mex(\{0,1,2\})=3, \nonumber
\end{align}
\end{proof}

\section{Conjectures}\label{conject}
In this section, we present some important conjectures that seem difficult to prove. The authors formulated these conjectures based on computer calculations.

\subsection{Prediction on the Reverse-Mex Property}
In this subsection, we present a prediction on the sufficient condition that a subtraction Nim has the reverse-$mex$ property.

We consider a subtraction Nim with the subtraction set $S=\{s_1, s_2, s_3\}$ satisfying $s_1<s_2 < s_3,$ where $ s_1, s_2, s_3 \in \mathbb{N}$. We suppose that there exist $p, q \in  \mathbb{N}$ such that $\mathcal{G}(x+p)=\mathcal{G}(x)$ for any $x \geq q+1$, i.e., the sequence of Grundy numbers $\mathcal{G}(x)$ has a loop whose length is $p$ for any $x \geq q+1$.
\begin{definition}
For $w \in \mathbb{Z}_{\geq0}$ and $x \in  \mathbb{N}$, let
\begin{equation} 
w- x \mod p =
\begin{cases}
w-x & (\mbox{ if } w-x \geq q+1 ),\\ \nonumber
w-x+np & (\mbox{ if } q+1 \leq w-x+np \leq q+p \mbox{ for some } n \in \mathbb{N}).\nonumber
\end{cases}
\end{equation} 
\end{definition}

\begin{definition}
For a $\mathcal{P}$-position $(w)$ , let
\begin{equation}
dist(w)= \min(\{k:(w-ks_1) \mod p \text{ is a } \mathcal{P}\text{-position} \}).
\end{equation}
\end{definition}

\begin{conjecture}\label{conditiona}
A subtraction Nim has the reverse-$mex$ property if and only if 
this Nim satisfies the following $(a)$.\\
$(a)$  For any $\mathcal{P}$-position $(w)$ such that  $dist(w)=2m$ for $m  \in \mathbb{N}$ with $m \geq 2$, we obtain
$w+s_3-s_1 \mod p$ is not a 
$\mathcal{P}$-position and  $w+s_3-3s_1 \mod p$ is a 
$\mathcal{P}$-position.
\end{conjecture}

\begin{conjecture}
If $s_3=s_1+s_2$, then the subtraction Nim satisfies the condition $(a)$ in Conjecture \ref{conditiona}. 
\end{conjecture}

 \subsection{Prediction on Reverse-$Mex$ Property for Subtraction Nim with a Pass}
 Here, we present examples of subtraction Nim with subtraction set  $\{s_1,s_2,s_3\}$ 
 that satisfy the following (i), (ii) and (iii). We allow a pass.\\
 (i) The sequence of Grundy numbers has the reverse-$mex$ property, i.e.,
\begin{equation}
\mathcal{G}(x,0)=mex(\{\mathcal{G}(x+s_1,0), \mathcal{G}(x+s_2,0), \mathcal{G}(x+s_3,0)\}).\label{conjecture1}
\end{equation}
(ii)  The Grundy number $\mathcal{G}(x,1)$ with a pass can be calculated without considering the value of the Grundy number $\mathcal{G}(x,0)$ without a pass, i.e., 
\begin{equation}
\mathcal{G}(x,1)=mex(\{\mathcal{G}(x-s_1,1), \mathcal{G}(x-s_2,1), \mathcal{G}(x-s_3,1)\}).\label{conjecture2}   
\end{equation}
(iii) The sequence of Grundy numbers $\mathcal{G}(x,1)$ has the reverse-$mex$ property, i.e.,
\begin{equation}
\mathcal{G}(x)=mex(\{\mathcal{G}(x+s_1,1), \mathcal{G}(x+s_2,1), \mathcal{G}(x+s_3,1)\}).\label{conjecture3}    
\end{equation}

\begin{conjecture}
$(a)$ For a subtraction Nim with a subtraction set $\{s_1,s_2,s_3\}$  $=\{a,2an,2an+a\}$, where $a, n \in \mathbb{N}$, we obtain 
(\ref{conjecture1}), and for $x \geq 6an+(4a+1) $, we have $(\ref{conjecture2})$ and $(\ref{conjecture3})$. 
The sequence of Grundy numbers enters into a loop when $x \geq 6an+(4a+1)$.\\
$(b)$ For a subtraction Nim with a subtraction set 
 $\{s_1,s_2,s_3\}$ $=\{a,(2n+1)a,(2n+3)a\}$, where $a, n \in \mathbb{N}$, we obtain 
(\ref{conjecture1}), and for $x \geq a(2n+3)+1$, we have $(\ref{conjecture2})$ and $(\ref{conjecture3})$. 
In this case, for any $x$, 
$\mathcal{G}(x,1)<4$. The sequence of Grundy numbers enters into a loop when $x \geq a(2n+3)+1$.\\
(iii) For a subtraction Nim with a subtraction set 
 $\{s_1,s_2,s_3\}$ $=\{a,(2n+1)a,(2n+5)a\}$, where $a, n \in \mathbb{N}$, 
 we obtain Equation $(\ref{conjecture1})$.
 In this case, for any $x$, $\mathcal{G}(x,1)<4$. 
For $x \geq a(2n+5)+1$, we obtain Equations $(\ref{conjecture2})$ and $(\ref{conjecture3})$.
In this case, for any $x$, the sequence of Grundy numbers enters into a loop when $x \geq a(2n+5)+1$.
\end{conjecture}

\end{document}